\documentclass[1p,final]{elsarticle}
\usepackage{amsfonts,color,morefloats,pslatex}
\usepackage{amssymb,amsthm, amsmath,latexsym}

\newtheorem{theorem}{Theorem}
\newtheorem{lemma}[theorem]{Lemma}

\newtheorem{corollary}[theorem]{Corollary}

\newtheorem{open}{Open Problem}

\newtheorem{conj}{Conjecture}

\newcommand{\tr}{{\mathrm{Tr}}}

\newcommand{\gf}{{\mathrm{GF}}}
\newcommand{\PG}{{\mathrm{PG}}}
\newcommand{\AG}{{\mathrm{AG}}}

\newcommand{\GA}{{\mathrm{GA}}} 
 
\newcommand{\GaA}{{\mathrm{\Gamma A}}}

\newcommand{\PGL}{{\mathrm{PGL}}}

\newcommand{\VS}{{\mathrm{VS}}}

\newcommand{\Stab}{{\mathrm{Stab}}}

\newcommand{\cA}{{\mathcal{A}}} 
   
\newcommand{\cP}{{\mathcal{P}}} 
\newcommand{\cB}{{\mathcal{B}}}

\newcommand{\cS}{{\mathcal{S}}} 
 
\newcommand{\cH}{{\mathcal{H}}}

\newcommand{\cW}{{\mathcal{W}}}

\newcommand{\bD}{{\mathbb{D}}}

\newcommand{\Segre}{ {{\mathrm{Segre}}} }
\newcommand{\Glynnone}{ {{\mathrm{Glynni}}} }
\newcommand{\Glynntwo}{ {{\mathrm{Glynnii}}} }
\newcommand{\Payne}{ {{\mathrm{Payne}}} }
\newcommand{\Trans}{ {{\mathrm{Trans}}} }

\newcommand{\Cherowitzo}{ {{\mathrm{Cherowitzo}}} } 
\newcommand{\Subiaco}{ {{\mathrm{Subiaco}}} }

\begin{document}

\begin{frontmatter}



\title{Combinatorial $t$-designs from special polynomials  
\tnotetext[fn1]{C. Ding's research was supported by the Hong Kong Research Grants Council,
Proj. No. 16300415. C. Tang was supported by National Natural Science Foundation of China (Grant No.
11871058) and China West Normal University (14E013, CXTD2014-4 and the Meritocracy Research
Funds).}
}

\author[cding]{Cunsheng Ding}
\ead{cding@ust.hk}
\author[cmt]{Chunming Tang}
\ead{tangchunmingmath@163.com}

\address[cding]{Department of Computer Science and Engineering, The Hong Kong University of Science and Technology, Clear Water Bay, Kowloon, Hong Kong, China}
\address[cmt]{School of Mathematics and Information, China West Normal University, Nanchong, Sichuan,  637002, China}

\begin{abstract}
Combinatorial $t$-designs have nice applications in coding theory, finite geometries and 
several engineering areas. There are two major methods of constructing $t$-designs. 
One of them is via group actions of certain permutation groups which are $t$-transitive 
or $t$-homogeneous on some point set. The other is a coding-theoretical one. The  
objectives of this paper are to introduce two constructions of $t$-designs with special 
polynomials over finite fields $\gf(q)$, and obtain $2$-designs and $3$-designs with 
interesting parameters. A type of d-polynomials is defined and used to construct $2$-designs. 
Under the framework of the first construction, it is shown that 
every o-polynomial over $\gf(2^m)$ gives a $2$-design, and every o-monomial over 
$\gf(2^m)$ yields a $3$-design. Under the second construction, every $o$-polynomial 
gives a $3$-design. Some open problems and conjectures are also presented in this paper. 
\end{abstract}

\begin{keyword}
Hyperoval \sep o-polynomial \sep polynomial \sep projective plane \sep $t$-design.

\MSC  51E21 \sep 05B05 \sep 12E10 

\end{keyword}

\end{frontmatter}

\section{Introduction}

Let $\cP$ be a set of $v \ge 1$ elements, and let $\cB$ be a set of $k$-subsets of $\cP$, where $k$ is
a positive integer with $1 \leq k \leq v$. Let $t$ be a positive integer with $t \leq k$. The pair
$\bD = (\cP, \cB)$ is called a $t$-$(v, k, \lambda)$ {\em design\index{design}}, or simply {\em $t$-design\index{$t$-design}}, if every $t$-subset of $\cP$ is contained in exactly $\lambda$ elements of
$\cB$. The elements of $\cP$ are called points, and those of $\cB$ are referred to as blocks.
We usually use $b$ to denote the number of blocks in $\cB$.  A $t$-design is called {\em simple\index{simple}} if $\cB$ does not contain repeated blocks. In this paper, we consider only simple 
$t$-designs.  A $t$-design is called {\em symmetric\index{symmetric design}} if $v = b$. It is clear that $t$-designs with $k = t$ or $k = v$ always exist. Such $t$-designs are {\em trivial}. In this paper, we consider only $t$-designs with $v > k > t$.
A $t$-$(v,k,\lambda)$ design is referred to as a {\em Steiner system\index{Steiner system}} if $t \geq 2$ and $\lambda=1$, and is denoted by $S(t,k, v)$.

By a special polynomial over a finite field we mean a polynomial either of special form or 
with special property. For instance, monomials and permutation polynomials are special 
polynomials. Special polynomials have interesting applications in combinatorial 
designs. For instance, the Dickson polynomials $x^5+ax^3+a^2x$ over $\gf(3^m)$ led to 
a 70-year breakthrough in searching for new skew Hadamard difference sets \cite{DY06}. 

A hyperoval in the projective space $\PG(2, \gf(2^m))$ is a set of $2^m+2$ points 
such that no three of them are collinear. O-polynomials are a special type of 
polynomials over $\gf(2^m)$ and correspond to hyperovals in the projective space 
$\PG(2, \gf(2^m))$ (see Theorem \ref{thm-hyperovaloply}). Hence, an o-polynomial can 
be viewed as a hyperoval, and vice versa. Hyperovals were used to construct two types 
of $2$-designs in the literature (see Theorems \ref{thm-hpodesign1} and \ref{thm-hpodesign2}). 
This means that o-polynomials can be used to construct $2$-designs indirectly (via 
their corresponding hyperovals). Motivated by this fact, in this paper we present two  
constructions of $t$-designs using o-polynomials directly. Specifically, we obtain 
$3$-designs from o-monomials and $2$-designs from general o-polynomials using the 
first construction. We obtain $3$-designs from o-polynomials using the second construction. 
We also introduce new types of polynomials over finite fields, which give also $2$-designs. 

The rest of this paper is arranged as follows. Section \ref{sec-construct} presents  
the general construction of $t$-designs from polynomials over general finite fields, 
and introduces a special type of polynomials, called \emph{design polynomial} 
(in short, d-polynomials). Section \ref{sec-designdpoly} investigates the designs 
of d-polynomials over $\gf(2^m)$. Section \ref{sec-designoploly} studies $3$-designs 
from o-monomials and $2$-designs from o-polynomials. Section \ref{sec-designgfq} 
deals with designs from polynomials over $\gf(q)$ for odd $q$. 
Section \ref{sec-extdendedconstr} introduces an extended construction of the one 
documented in Section \ref{sec-construct}, and deals with $2$-designs and 
$3$-designs via this extended construction.  
Section \ref{sec-summary} 
concludes this paper and makes concluding remarks.

\section{A construction for $t$-designs with polynomials over $\gf(q)$}\label{sec-construct} 

Let $q$ be a prime power, and let $f$ be a polynomial over $\gf(q)$, which is always viewed 
as a function from $\gf(q)$ to $\gf(q)$ throughout this paper. For each $(b, c) \in \gf(q)^2$, 
define 
\begin{eqnarray}
B_{(f,b,c)}=\{f(x)+bx+c:  x \in \gf(q)\}. 
\end{eqnarray} 
Let $k$ be an integer with $2 \leq k \leq q$. Define 
\begin{eqnarray}
\cB_{(f,k)}=\{B_{(f,b,c)}: |B_{(f,b,c)}|=k, \ b, \ c \in \gf(q)\}. 
\end{eqnarray} 
The incidence structure $\bD(f, k):=(\gf(q), \cB_{(f,k)})$ may be a $t$-$(q, k, \lambda)$ design for some $\lambda$, 
where $\gf(q)$ is the point set, and the incidence relation is the set membership. In this 
case, we say that the polynomial $f$ supports a $t$-$(q, k, \lambda)$ design.

The following is a general result about monomials. It shows an interesting application 
of monomials in the theory of combinatorial designs. 

\begin{theorem}\label{thm-380}
Let $f(x)=x^e$ be a permutation polynomial of $\gf(q)$, and let $k \geq 2$ be a positive integer 
such that $|\cB_{(f,k)}| \geq 1$. Then the incidence structure $\bD(f, k):=(\gf(q), \cB_{(f,k)})$ is a $2$-$(q, k, \lambda)$ for some $\lambda$.  
\end{theorem} 

\begin{proof} 
The general affine group $\GA_1(\gf(q))$ is defined by 
$$ 
\GA_1(\gf(q)):=\{ux+v: (u, v) \in \gf(q)^* \times \gf(q)\}.  
$$ 
Let $\sigma(x)=ux+v \in \GA_1(\gf(q))$, where $u \in \gf(q)^*$ and $v \in \gf(q)$. 
Note that $\gcd(e, q-1)=1$. Let $1/e$ denote the multiplicative inverse of $e$ modulo 
$q-1$. We have then  
\begin{eqnarray*}
u(f(x)+bx+c)+v 
&=& ux^e+ubx+cu + v \\ 
&=& (u^{1/e}x)^e + u^{1-1/e}b(u^{1/e}x)+ cu+v.  
\end{eqnarray*} 
We then deduce that $\sigma(B_{(f,b,c)})=B_{(f, u^{1-1/e}b, cu+v)}$. This means that 
the general affine group $\GA_1(\gf(q))$ fixes $\cB_{(f,k)}$. It is well known that 
$\GA_1(\gf(q))$ acts on $\gf(q)$ doubly transitively. The desired conclusion then 
follows. 
\end{proof} 

Two designs $\bD(\cP, \cB)$ and $\bD(\cP', \cB')$ are said to be \emph{isomorphic} if 
there is a 1-to-1 mapping $\sigma$ from $\cP$ to $\cP'$ such that $\sigma$ sends 
each block in $\cB$ to a block in $\cB'$. Such a $\sigma$ is called an isomorphism 
from $\bD(\cP, \cB)$ to $\bD(\cP', \cB')$. An isomorphism from $\bD(\cP, \cB)$ to 
$\bD(\cP, \cB)$ is called an \emph{automorphism} of $\bD(\cP, \cB)$. All automorphisms 
of $\bD(\cP, \cB)$ form a group under the function composition, and is called the 
automorphism group of $\bD(\cP, \cB)$. It is straightforward to prove the following theorem. 

\begin{theorem}\label{thm-isomorphy1}
Let $f$ and $g$ be two polynomials over $\gf(q)$ such that $\bD(f, k)$ and $\bD(g,k)$ 
are $t$-designs. If there are $h \in \gf(q)^*$, $u\in \gf(q)^*$ and $v \in \gf(q)$ 
such that $g(x)=hf(ux+v)$ for all $x \in \gf(q)$, then $\bD(f, k)$ and $\bD(g,k)$ are isomorphic. 
\end{theorem} 

We define the \emph{value spectrum} of a polynomial over $\gf(q)$ to be the multiset 
\begin{eqnarray*}
\VS(f)=\{\{ |B_{(f,b,c)}|: (b,c) \in \gf(q)^2 \}\}.  
\end{eqnarray*} 
To determine the parameters of $t$-designs supported by a polynomial $f$, we need to 
know the value spectrum of a polynomial $f$. The value spectrum of a polynomial is hard 
to determine in general, but can be done in special cases. 

We call a permutation polynomial of $\gf(q)$ a \emph{design polynomial} (in short, 
d-polynomial) if the size $|\{f(x)+bx\}|$ is a constant for all $b \in \gf(q)^*$. 
As will be seen later, some d-polynomials supports $2$-designs or $3$-designs with 
interesting parameters under this construction framework.

\section{Designs from d-monomials over $\gf(2^m)$}\label{sec-designdpoly} 

Throughout this section, let $q=2^m$ for some positive integer $m$. Our objective in 
this section is to search for d-monomials and consider the parameters of their 
$2$-designs. As will be seen soon, determining the block size and the number of blocks 
in the $2$-design supported by a d-monomial could be extremely hard. There are a 
number of such d-monomials. Some of them are treated in this section, and some will 
be investigated in Section \ref{sec-designoploly}.

\begin{lemma}\label{lem-381}
Let $f(x)=x^e$ be a polynomial over $\gf(q)$ such $\gcd(e(e-1), q-1)=1$. 
Then $f(x)$ is a d-polynomial over $\gf(q)$. 
\end{lemma} 

\begin{proof}
Since $x^{e-1}$ is a permutation of $\gf(q)$, for each $b \in \gf(q)^*$ there is a unique 
$u \in \gf(q)^*$ such that $u^{e-1}=b$. We have then 
\begin{eqnarray*}
|\{x^e+bx: x \in \gf(q)\}| 
&=& |\{(uy)^e+buy: y \in \gf(q)\}| \\
&=& |\{u^e(y^e+y): y \in \gf(q)\}| \\ 
&=& |\{y^e+y: y \in \gf(q)\}|. 
\end{eqnarray*} 
By definition, $f(x)=x^e$ is a d-polynomial. 
\end{proof} 

There are a number of monomials $f(x)=x^e$ satisfying $\gcd(e(e-1), q-1)=1$. 
Such $d$-monomials over $\gf(q)$ support $2$-designs by Theorem \ref{thm-380}. 
It will be shown later that some of them support $3$-designs.

The following theorem presents a large number of $2$-designs supported by a class of d-monomials. 

\begin{theorem}\label{thm-main2design1}
Let $e$ be a positive integer with $\gcd(e(e-1), q-1)=1$. Define 
\begin{eqnarray}\label{eqn-Je1}
J_e=\{x^e+x: x \in \gf(q)\} 
\end{eqnarray} 
and 
$$ 
\Stab_{\AG_1(\gf(q))}(J_e)=\{ux+v: (u,v) \in \gf(q)^* \times \gf(q), \ uJ_e+v=J_e\}. 
$$ 
Then the incidence structure $\bD(x^e, k):=(\gf(q), \cB_{(x^e,k)})$ is a $2$-$(q, k, (k-1)k/\mu)$ design, where $k=|J_e|$ and 
$\mu=|\Stab_{\AG_1(\gf(q))}(J_e)|$. 
\end{theorem} 

\begin{proof}
It follows from Lemma \ref{lem-381} and its proof that $f(x)=x^e$ is a d-monomial 
and that 
$$ 
\cB_{(x^e,k)}=\{bJ_e+c: (b,v) \in \gf(q)^* \times \gf(q) \}. 
$$ 
Note that the general affine group acts on $\gf(q)$ doubly transitively and fixes 
$\cB_{(x^e,k)}$. The desired 
conclusion then follows from Proposition 4.6 in \cite[p. 175]{BJL}. 
\end{proof}

\begin{theorem}\label{thm-omonomials1}
Let $m \geq 3$ be odd and $q=2^m$. 
The following is a list of d-monomials $x^e$ over $\gf(q)$ satisfying the condition of Theorem 
\ref{thm-main2design1}. 
\begin{itemize}
\item $e=2^h+1$, where $\gcd(h, m)=1$. 
\item $e=2^{(m-1)/2}+3$. 
\item $e=2^{2h}-2^h+1$, where $\gcd(h, m)=1$. 
\item $e= 2^{(m-1)/2}+2^{(m-1)/4}-1$, where $m \equiv 1 \pmod{4} \geq 5$.  
\item $e= 2^{(m-1)/2}+2^{(3m-1)/4}-1$, where $m \equiv 3 \pmod{4}$.
\end{itemize} 
\end{theorem}

\begin{proof}
In all the cases above, it can be verified that $\gcd(e(e-1), q-1)=1$. It then 
follows from Lemma \ref{lem-381} that $x^e$ is a d-polynomial.  
\end{proof}

\begin{theorem}\label{thm-omonomials2}
Let $m$ be even and $q=2^m$. 
The following is a list of d-monomials $x^e$ over $\gf(q)$ satisfying the condition of Theorem 
\ref{thm-main2design1}. 
\begin{itemize} 
\item $e=2^h+1$, where $m/\gcd(h, m)$ is odd. 
\item $e=2^{m/2} + 2^{(m+2)/4}+1$, where $m \equiv 2 \pmod{8}$. 
\item $e=2^{(m-2)/2} -3$, where $m \equiv 4 \pmod{24}$ or $m \equiv 20 \pmod{24}$. 
\item $e=2^{(m+2)/2} -3$, where $m \equiv 0 \pmod{24}$ or $m \equiv 8 \pmod{24}$ or 
      $m \equiv 16 \pmod{24}$.  
\end{itemize} 
\end{theorem} 

\begin{proof}
In all the cases above, it can be verified that $\gcd(e(e-1), q-1)=1$. It then 
follows from Lemma \ref{lem-381} that $x^e$ is a d-polynomial. 
\end{proof}

All the d-monomials in Theorems \ref{thm-omonomials1} and \ref{thm-omonomials2} can 
be plugged into Theorem \ref{thm-main2design1} for obtaining $2$-$(q, k, \lambda)$ 
designs. But determining the parameters $k$ and $\lambda$ in the designs seems quite 
difficult. The reader is warmly invited to attack this problem.

\section{Designs from o-polynomials over $\gf(2^m)$}\label{sec-designoploly} 

Throughout this section $q=2^m$ for some positive integer $m$. The objective of this section 
is to construct 2-designs and $3$-designs from o-polynomials over $\gf(q)$. Since o-polynomials and hyperovals can be viewed as the same and hyperovals were used to construct two types of $2$-designs 
in the literature, we have to introduce hyperovals and their designs, so that we will be able 
to compare our newly constructed designs with hyperoval designs in the literature. 

\subsection{Hyperovals and their designs} 

An \emph{arc}\index{arc} in the projective plane $\PG(2, \gf(q))$ is a set of at least three 
points in $\PG(2, \gf(q))$ such that no three of them are collinear (i.e., on the same line). For any arc $\cA$ of $\PG(2, \gf(q))$, it is well known that $|\cA| \leq q+2$. 

A \emph{hyperoval} $\cH$ in $\PG(2,\gf(q))$ is a set of $q+2$ points such that no three of 
them are collinear, i.e., an arc with $q+2$ points. Hyperovals are maximal arcs, as they have the maximal number of points as arcs.
Two hyperovals  
are said to be \emph{equivalent} if there is a collineation (i.e., 
an automorphism) of $\PG(2, \gf(q))$ that sends one to the other. 
Note that the automorphism group of $\PG(2, \gf(q))$ is the  
projective general linear group $\PGL_3(\gf(q))$. 
The \emph{automorphism group}    
of a hyperoval is the set of all collineations of $\PG(2, \gf(q))$ that leave 
the hyperoval invariant.     

The 
next theorem shows that all hyperovals in $\PG(2,\gf(q))$ can be constructed with 
a special type of permutation polynomials of $\gf(q)$ \cite[p. 504]{LN97}. 

\begin{theorem}[Segre]\label{thm-hyperovaloply}
Let $m \geq 2$. Any hyperoval in $\PG(2, \gf(q))$ can be written in the form 
$$ 
\cH(f)=\{(f(c), c, 1): c \in \gf(q)\} \cup \{(1,0,0)\}  \cup \{(0,1,0)\},    
$$ 
where $f \in \gf(q)[x]$ is such that 
\begin{enumerate} 
\item $f$ is a permutation polynomial of $\gf(q)$ with $\deg(f)<q$ and $f(0)=0$, $f(1)=1$;  
\item for each $a \in \gf(q)$, $g_a(x)=(f(x+a)+f(a))x^{q-2}$ is also a permutation polynomial 
      of $\gf(q)$. 
\end{enumerate} 
Conversely, every such set $\cH(f)$ is a hyperoval. 
\end{theorem}  

Polynomials satisfying the two conditions of Theorem \ref{thm-hyperovaloply} are 
called \emph{o-polynomials}\index{o-polynomial}, i.e., oval-polynomials. For example, 
$f(x)=x^2$ is an o-polynomial over $\gf(q)$ for all $m \geq 2$. In the next section, we will 
summarize known o-polynomials over $\gf(q)$. 

Two o-monomials $f$ and $g$ are said to be equivalent if the two hyperovals $\cH(f)$ 
and $\cH(g)$ are equivalent. The following result was presented in \cite{Xiang}. 

\begin{lemma}\label{lem-GEHKX} 
Let $q \geq 4$. Two monomial hyperovals $\cH(x^j)$ and $\cH(x^e)$ 
in $\PG(2, \gf(q))$ are equivalent if and only if $i \equiv e, 1/e, 1-e, 1/(1-e), 
e/(e-1) \mbox{ or } (e-1)/e \pmod{q-1}$.  
\end{lemma}

Any hyperoval $\cH$ in $\PG(2, \gf(q))$ meets each line either in $0$ or $2$ points. 
A line is called an interior line (also called secant) of $\cH$ if it meets the hyperoval in two points, and 
an exterior line otherwise. 
Hence, a hyperoval partitions the lines of      
$\PG(2, \gf(q))$ into two classes, i.e., interior and exterior lines. This property allows us to define the so-called 
hyperoval designs as follows. 
 
Let $\cH$ be a hyperoval in the Desarguesian projective plane $\PG(2, \gf(q))$. The \emph{hyperoval 
design}\index{hyperoval design} $\cW(q, \cH)$ is the incidence structure with points the 
lines of $\PG(2, \gf(q))$ exterior to $\cH$ and blocks the points of $\PG(2, \gf(q))$ not 
on the hyperoval; incidence is given by the incidence in $\PG(2, \gf(q))$. We have then the following 
conclusion on the incidence structure $\cW(q, \cH)$. 

\begin{theorem}[\cite{AK92}]\label{thm-hpodesign1}
The incidence structure $\cW(q, \cH)$ defined by a hyperoval $\cH$ in $\PG(2, \gf(q))$ 
is a $2$-$((q-1)q/2, q/2, 1)$ design, i.e., a Steiner system\index{Steiner system}.  
\end{theorem} 

The second type of $2$-designs from hyperovals are constructed as follows. 
Let $\cH$ be a hyperoval in $\PG(2, \gf(q))$. Let $\cP$ 
be the set of $q^2-1$ exterior points to $\cH$, i.e., the set of points in $\PG(2, \gf(q)) \setminus 
\cH$. For each point $x \in \cP$, define a block 
$$ 
B_x=\{y \in \cP\setminus \{x\}: xy \mbox{ is a secant to } \cH\} \cup \{x\}.  
$$ 
Define further $\cB=\{B_x: x \in \cP\}$. We have then the following conclusion.  

\begin{theorem}[\cite{AK92,Jackson, Masch98, Pott}]\label{thm-hpodesign2}
The incidence structure $\cS(q, \cH):=(\cP, \cB)$ is a symmetric 
$2$-$(q^2-1, \frac{1}{2}q^2-1, \frac{1}{4}q^2-1)$ design. 
\end{theorem} 

It is known that the Hadamard design $\cS(q, \cH)$ can be extended into a $3$-$(q^2, \frac{1}{2}q^2, \frac{1}{4}q^2-1)$ design, denoted by $\cS(q, \cH)^e$ \cite{AK92}.

\subsection{Known o-polynomials over $\gf(2^m)$}

Recall that $q=2^m$. 
To construct $2$-designs and $3$-designs subsequently, we need o-polynomials over $\gf(q)$. The objective of this section 
is to summarise known constructions of o-polynomials over $\gf(q)$ and consequently hyperovals in 
$\PG(2, \gf(q))$. 

In the definition of o-polynomials, it is required that $f(1)=1$. However, this is not essential, as 
one can always normalise $f(x)$ by using $f(1)^{-1} f(x)$ due to the fact that $f(1) \neq 0$. In this section, 
we do not require that $f(1)=1$ for o-polynomials.  

For any permutation polynomial $f(x)$ over $\gf(q)$, we define $\overline{f}(x)=xf(x^{q-2})$, and use $f^{-1}$ to denote the compositional inverse of $f$, i.e., $f^{-1}(f(x))=x$ for all $x \in \gf(q)$.

The following two theorems introduce basic properties of o-polynomials whose proofs can be 
found in references about hyperovals. 

\begin{theorem}\label{thm-basicproperty}
Let $f$ be an o-polynomial over $\gf(q)$. Then the following statements hold: 
\begin{itemize}
\item $f^{-1}$ is also an o-polynomial; 
\item $f(x^{2^{j}})^{2^{m-j}}$ is also an o-polynomial for any $1 \leq j \leq m-1$; 
\item $\overline{f}$ is also an o-polynomial; and  
\item $f(x+1)+f(1)$ is also an o-polynomial. 
\end{itemize}
\end{theorem} 

\begin{theorem}\label{thm-basicproperty2}
Let $x^e$ be an o-polynomial over $\gf(q)$. Then every polynomial in 
$$\left\{x^{\frac{1}{e}},\, x^{1-e},\,  x^{\frac{1}{1-e}},\, x^{\frac{e}{e-1}},\, x^{\frac{e-1}{e}}\right\}$$ 
is also an o-polynomial, where $1/e$ denotes the multiplicative inverse of $e$ modulo $q-1$. 
\end{theorem} 

\begin{theorem}[\cite{Masch98}]\label{thm-opoly2to1}
A polynomial $f$ over $\gf(q)$ with $f(0)=0$ is an o-polynomial if and only if $f_u:=f(x)+ux$ 
is $2$-to-$1$ for every $u \in \gf(q)^*$. 
\end{theorem}

Below we summarise some classes of o-polynomials over $\gf(q)$.  
The translation o-polynomials are described in the following theorem \cite{Segre57}. 

\begin{theorem}
$\Trans(x)=x^{2^h}$ is an o-polynomial over $\gf(q)$, where $\gcd(h, m)=1$.
\end{theorem}

The following is a list of known properties of translation o-polynomials.  
\begin{itemize}
\item $\Trans^{-1}(x)=x^{2^{m-h}}$ and  
\item $\overline{\Trans}(x)= xf(x^{q-2})=x^{q-2^{m-h}}$. 
\end{itemize}

The following theorem describes a class of o-polynomials, which are called Segre o-polynomials \cite{Segre62,SegreBartocci}. 

\begin{theorem} 
Let $m$ be odd. Then 
$\Segre(x)=x^{6}$ is an o-polynomial over $\gf(q)$.
\end{theorem}

For this o-monomial, we have the following.  
\begin{enumerate} 
\item $\overline{\Segre}(x)=x^{q-6}$. 
\item $\Segre^{-1}(x)=x^{\frac{5\times 2^{m-1}-2}{3}}$. 
\end{enumerate}

Glynn discovered two families of o-polynomials \cite{Glynn83}. The first is described as follows. 

\begin{theorem}
Let $m$ be odd. Then $\Glynnone(x)=x^{3\times 2^{(m+1)/2}+4}$ is an o-polynomial.
\end{theorem}

The second family of o-polynomials discovered by Glynn is documented in the following theorem. 

\begin{theorem} 
Let $m$ be odd. Then 
\begin{eqnarray*}
\Glynntwo(x)= \left\{ 
\begin{array}{r}
x^{2^{(m+1)/2}+2^{(3m+1)/4}}   \mbox{ if } m \equiv 1 \pmod{4}, \\
x^{2^{(m+1)/2}+2^{(m+1)/4}}    \mbox{ if } m \equiv 3 \pmod{4}          
\end{array}
\right. 
\end{eqnarray*}
 is an o-polynomial over $\gf(q)$. 
\end{theorem}

The following describes another class of o-polynomials discovered by Cherowitzo  \cite{Ch88,Ch96}.  

\begin{theorem}[\cite{DingYuan}]
Let $m$ be odd and $e=(m+1)/2$. Then 
$$
\Cherowitzo(x)=x^{2^e}+x^{2^e+2}+x^{3 \times 2^e +4} 
$$ 
is an o-polynomial over $\gf(q)$.
\end{theorem}

For this o-trinomial, we have the following conclusions.  
\begin{enumerate}
\item $\overline{\Cherowitzo}(x)=x^{q-2^e}+x^{q-2^e-2}+x^{q-3 \times 2^e -4}$. 
\item $\Cherowitzo^{-1}(x)=x(x^{2^e+1}+x^3+x)^{2^{e-1}-1}$. 
\end{enumerate}

The following  documents a family of o-trinomials due to Payne. 

\begin{theorem}[\cite{Pay85}]
Let $m$ be odd. Then 
$\Payne(x)=x^{\frac{5}{6}}+x^{\frac{3}{6}}+x^{\frac{1}{6}}$ is an o-polynomial over $\gf(q)$.
\end{theorem}

We have the following statements regarding the Payne o-trinomial.  
\begin{enumerate}
\item $\Payne(x)=xD_5(x^{\frac{1}{6}}, 1)$, where $D_5(a, x)$ is the Dickson polynomial of order $5$.  
\item $\overline{\Payne}(x)=\Payne(x)$. 
\item Note that 
$$ 
\frac{1}{6}=\frac{5 \times 2^{m-1}-2}{3}. 
$$
We have then 
$$ 
\Payne(x)=x^{\frac{2^{m-1}+2}{3}} + x^{2^{m-1}} + x^{\frac{5 \times 2^{m-1}-2}{3}}. 
$$
\end{enumerate}

\begin{theorem}[\cite{DingYuan}]
Let $m$ be odd. Then 
\begin{eqnarray}\label{eqn-PayneInverse}
\Payne^{-1}(x)=\left( D_{\frac{3 \times 2^{2m}-2}{5}}(x, 1)\right)^6. 
\end{eqnarray}  
\end{theorem}

The Subiaco o-polynomials are given in the following theorem \cite{Subiaco}. 

\begin{theorem}\label{thm-Subiaco}
Define 
$$ 
\Subiaco_a(x)=((a^2(x^4+x)+a^2(1+a+a^2)(x^3+x^2)) (x^4 + a^2 x^2+1)^{q-2}+x^{2^{m-1}},
$$ 
where $\tr(1/a)=1$ and $d \not\in \gf(4)$ if $m \equiv 2 \bmod{4}$. Then $\Subiaco_a(x)$ is an 
o-polynomial over $\gf(q)$. 
\end{theorem}

As a corollary of  Theorem \ref{thm-Subiaco}, we have the following. 

\begin{corollary} 
Let $m$ be odd. Then 
\begin{eqnarray}\label{cor-Subiaco}
\Subiaco_1(x)=(x+x^2+x^3+x^4) (x^4 + x^2+1)^{q-2}+x^{2^{m-1}} 
\end{eqnarray}
is an o-polynomial over $\gf(q)$. 
\end{corollary}

\subsection{Combinatorial $t$-designs from o-polynomials} 

In this section, we plug o-polynomials into the construction of Section \ref{sec-construct}
to construct $2$-designs and $3$-designs. By Theorem \ref{thm-opoly2to1}, o-polynomials are 
d-polynomials. This fact will play an important role. 

\subsubsection{Families of $2$-designs and $3$-designs from o-polynomials}

We start with a few auxiliary results. Let
$g(x)$  be a  polynomial over $\mathrm{GF}(q)$. The value set of 
$g(x)$ is the image of the induced map $g: \mathrm{GF}(q) \mapsto \mathrm{GF}(q)$. 
Thus the value set is
 \begin{align*}
 V(g)=\{g(x): x\in \mathrm{GF}(q)\}.
 \end{align*}
We  denote the cardinality of $V(g)$ by $v(g)$. 
 
\begin{lemma}\label{lem-mar91}
Let $f(x)\in \mathrm{GF}(q)[x]$ be an o-polynomial. For any $u_1, u_2, u_3 \in \mathrm{GF}(q)$ with $(u_1+u_2)(u_2+u_3)(u_3+u_1)  \neq  0$, define
 \begin{align*}
 I(u_1,u_2,u_3)=\left \{(a,b,c)\in \mathrm{GF}(q)^3: ab\neq 0, \{u_1, u_2, u_3\} \subseteq  V(af(x)+bx+c) \right \}.
 \end{align*}
Then $|I(u_1,u_2,u_3)|=\frac{q(q-1)(q-4)}{8}$.\\
\end{lemma} 

\begin{proof}
Put 
\begin{align*}
T=\left \{(a,b,c, x_1,x_2,x_3)\in \mathrm{GF}(q)^6:af(x_i)+bx_i+c=u_i \ (i=1,2,3) \right  \}.
\end{align*}
Then
\begin{align*}
|T| = \sum_{(a,b,c)\in \mathrm{GF}(q)^3} J(a,b,c)= \sum_{(x_1, x_2, x_3) \in \mathrm{GF}(q)^3} K(x_1,x_2,x_3),
\end{align*}
where
\begin{align*}
J(a,b,c)=|\left  \{(x_1, x_2, x_3)\in \mathrm{GF}(q)^3 : af(x_i)+bx_i+c=u_i \ (i=1,2,3)\right \}|,
\end{align*}
and 
\begin{align*}
K(x_1,x_2,x_3)=|\left  \{(a, b, c)\in \mathrm{GF}(q)^3 : af(x_i)+bx_i+c=u_i \ (i=1,2,3)\right \}|.
\end{align*}

Notice that $g(x)=af(x)+bx+c$ is 2-to-1 when $ab \neq 0$. We have $v(g)=q/2$ if $ab \neq 0$. 
If $ab=0$ and $a \neq b$, then $g(x)$ is a permutation. We deduce then  
\begin{align*}
v(af(x)+bx+c)=\begin{cases}
1, & \mbox{ if and only if } a=b=0,\\
 q, & \mbox{ if and only if } ab=0 \text{ and } a\neq b,\\
 q/2, & \mbox{ if and only if } ab\neq 0. 
\end{cases}
\end{align*}  

Since $g(x)=af(x)+bx+c$ is 2-to-1 when $v(g)=q/2$ and  is a permutation when $v(g)=q$, 
we have 
\begin{align*}
J(a,b,c)=\begin{cases}
0, & \mbox{ if } \{u_1, u_2, u_3\} \not \subseteq V(g),\\
 1, &  \mbox{ if } \{u_1, u_2, u_3\}  \subseteq V(g) \text{ and } v(g)=q,\\
 8, &  \mbox{ if } \{u_1, u_2, u_3\}  \subseteq V(g) \text{ and } v(g)=q/2. 
\end{cases}
\end{align*} 
It then follows that 
\begin{align*}
|T|= & \sum_{(a,b,c)\in \mathrm{GF}(q)^3} J(a,b,c)\\
=& |\{(a,b,c) \in \mathrm{GF}(q)^3 : v(af(x)+bx+c)=q\}| + 8 |I(u_1,u_2,u_3)| \\
=& 2(q-1)q + 8 |I(u_1,u_2,u_3)|.
\end{align*} 

Let $x_1, x_2$ and $x_3$ be three pairwise distinct elements in $\mathrm{GF}(q)$. 
Then $(f(x_1), x_1, 1)$, $(f(x_2), x_2, 1)$, and $(f(x_3), x_3, 1)$ are three points 
in the hyperoval defined by the o-polynomial $f(x)$, and thus are linearly independent 
over $\mathrm{GF}(q)$. We then deduce that 
\begin{align*}
K(x_1,x_2,x_3)=\begin{cases}
0, & | \{x_1, x_2, x_3\} | <3,\\
1, & | \{x_1, x_2, x_3\} | =3.
\end{cases}
\end{align*}
Thus,
\begin{align*}
|T| = & \sum_{(x_1, x_2, x_3)\in \mathrm{GF}(q)^3} K(x_1,x_2,x_3)= q(q-1)(q-2).
\end{align*}
Consequently, 
\begin{align*}
I(u_1,u_2,u_3)=\frac{1}{8}\left ( q(q-1)(q-2) -2(q-1)q \right ) =\frac{q(q-1)(q-4)}{8}.
\end{align*}
This completes the proof. 
\end{proof}

\begin{lemma}\label{lem-mar92} 
Let $a\in \mathrm{GF}(q)^*$ and  $f(x)= x^d \in \mathrm{GF}(q)[x]$ be an o-monomial. 
For any $u_1, u_2, u_3 \in \mathrm{GF}(q)$ with $(u_1+u_2)(u_2+u_3)(u_3+u_1)  \neq  0$, define
 \begin{align*}
 I_a(u_1,u_2,u_3)=\left \{(b,c)\in \mathrm{GF}(q)^2: b \neq 0, \{u_1, u_2, u_3\} \subseteq  V(af(x)+bx+c) \right \}.
 \end{align*}
Then, $ |I_a(u_1,u_2,u_3)| = \frac{q(q-4)}{8}$. 
\end{lemma} 

\begin{proof}
Recall that $f(x)=x^d$ is a permutation of $\mathrm{GF}(q)$. We have then 
\begin{align*}
V(af(x)+bx+c)=& \{a x^d +bx +c: x\in \mathrm{GF}(q)\}\\
=& \{ (a^{d^{-1}}x)^d + ba^{-d^{-1}} (a^{d^{-1}}x) +c: x \in \mathrm{GF}(q)\}\\
=& V(x^d+ ba^{-d^{-1}} x +c),
\end{align*}
where $d^{-1}$ is a positive integer such that  $d d^{-1} \equiv 1 \pmod {q-1}$.
Thus, $(b, c) \longmapsto ( ba^{-d^{-1}}, c)$ induces a bijective mapping from $I_a(u_1,u_2,u_3)$ to $I_1(u_1,u_2,u_3)$.
Then, $ | I_a(u_1,u_2,u_3) | = | I_1(u_1,u_2,u_3) |$.
We then deduce by Lemma \ref{lem-mar91} that 
\begin{align*}
| I_a(u_1,u_2,u_3) | =\frac{1}{q-1} | I(u_1,u_2,u_3)|=\frac{q(q-4)}{8}.
\end{align*}
This completes the proof. 
\end{proof}  

We are now ready to prove the following result, which is one of the main results of this paper. 

\begin{theorem}\label{thm-mar901}
Let $f(x)=x^e$ be an o-monomial over $\gf(q)$. 
Then $\bD(f, q/2):=(\gf(q), \cB_{(f,q/2)})$ is a $3$-$(q, q/2, q(q-4)/8\mu)$ design, where 
$$ 
\mu = \left| \Stab_{\AG_1(\gf(q))}(J_e) \right|  
    = |\{(u, v) \in \gf(q)^* \times \gf(q): \ uJ_e+v=J_e \}|
$$ 
and 
\begin{eqnarray}\label{eqn-Je2}
J_e=\{y^e+y: y \in \gf(q)\}. 
\end{eqnarray} 
\end{theorem}

\begin{proof}
We follow the notation of Lemmas \ref{lem-mar91} and \ref{lem-mar92} and their proofs. 
By the definition of o-polynomials, we have $\gcd(e(e-1), q-1)=1$. Define the following 
multiset: 
$$ 
\bar{\cB}_{(f,q/2)} = \{\{ \{x^e+bx+c: x \in \gf(q)\}: b \in \gf(q)^*, c \in \gf(q)\}\} 
$$ 
By the proof of Lemma \ref{lem-381}, 
$$ 
\bar{\cB}_{(f,q/2)} = \{\{ b J_e + c: b \in \gf(q)^*, c \in \gf(q)\}\} 
$$ 
and 
$$ 
\cB_{(f,q/2)} = \{ b J_e + c: b \in \gf(q)^*, c \in \gf(q)\}.  
$$ 
Clearly, the general affine group $\AG_1(\gf(q))$ fixes both $\bar{\cB}_{(f,q/2)}$ 
and $\cB_{(f,q/2)}$. The stabilizer of $J_e$ under $\AG_1(\gf(q))$ is defined by 
$$ 
\Stab_{\AG_1(\gf(q))}(J_e)=\{ux+v: (u,v) \in \gf(q)^* \times \gf(q), uJ_e+v=J_e\}.  
$$   
We then deduce that 
$$ 
|\cB_{(f,q/2)}|= \frac{(q-1)q}{|\Stab_{\AG_1(\gf(q))}(J_e)|}. 
$$ 

Note that 
$$ 
V(x^e+bx+c) = V(b^{e/(e-1)}(x^e+x)+c). 
$$
Consequently, the multiset 
$$ 
\{\{ V(x^e+bx+c): (b,c) \in I_1(u_1,u_2,u_3) \}\} 
$$ 
is the same as the multiset 
$$ 
|\Stab_{\AG_1(\gf(q))}(J_e)| \{\{ B_{(f, b, c)} \in \cB_{(f,q/2)}: \{u_1, u_2,u_3\} \subset B_{(f, b, c)} \}\},  
$$ 
where $\{u_1, u_2, u_3\}$ is a set of three distinct elements in $\gf(q)$, 
and $I_a(u_1, u_2, u_3)$ was defined in Lemmas \ref{lem-mar91} and \ref{lem-mar92}. 
It then follows that $(\gf(q), \bar{\cB}_{(f,q/2)})$ is a $t$-$(q, q/2, \lambda)$ design 
if and only if $(\gf(q), \cB_{(f,q/2)})$ is a $t$-$(q, q/2, \lambda/\mu)$ design, where 
$\mu$ was defined earlier. 

By Lemma \ref{lem-mar92}, $(\gf(q), \bar{\cB}_{(f,q/2)})$ is a $3$-$(q, q/2, q(q-4)/8)$ 
design, which may contain repeated blocks. As a result, $(\gf(q), \cB_{(f,q/2)})$ is a 
$3$-$(q, q/2, q(q-4)/8\mu)$ simple design.   
\end{proof}

Theorem \ref{thm-mar901} says that every o-monomial $x^e$ supports a $3$-design 
$\bD(x^e, q/2)$. The determination of the parameters of 
the $3$-design boils down to that of the size $\mu$ of the stabiliser of 
the block $J_e$ under the action of $\GA_1(\gf(q))$.

The following is a corollary of Theorem \ref{thm-mar901}. We give a direct proof 
of it below. 

\begin{corollary}\label{cor-mar93}
Let $f(x)=x^e$ be an o-monomial over $\gf(q)$ such that $|\cB_{(f,q/2)}|=(q-1)q$. 
Then $\bD(f, q/2):=(\gf(q), \cB_{(f,q/2)})$ is a $3$-$(q, q/2, q(q-4)/8)$ design. 
\end{corollary}

\begin{proof}
It follows from Theorem \ref{thm-opoly2to1} that $|B_{(f,b,c)}|=q/2$ for all 
$(b, c) \in \gf(q)^* \times \gf(q)$. By assumption, all blocks $B_{(f,b,c)}$ 
with $(b, c) \in \gf(q)^* \times \gf(q)$ are pairwise distinct. The design 
property then follows from Lemma \ref{lem-mar92}.  
\end{proof} 

Only o-monomials support $3$-designs with respect to this construction. 
O-polynomials do not support $3$-designs in general, but do support $2$-designs 
with respect to this construction. Below we prove this general result. To this end, we need prove 
the next two auxiliary results. 

\begin{lemma}\label{lem-mar94} 
 Let $f(x)\in \gf(q)[x]$ be an o-polynomial. For any $u_1, u_2 \in \mathrm{GF}(q)$ with $ u_1 \neq u_2$, define
 \begin{align*}
 I(u_1,u_2)=\left \{(b,c)\in \mathrm{GF}(q)^2: b\neq 0, \{u_1, u_2\} \subseteq  V(f(x)+bx+c) \right \}.
 \end{align*}
Then, $| I(u_1,u_2) | = \frac{q(q-2)}{4}$. 
\end{lemma} 

\begin{proof}
Set
\begin{align*}
T=\left \{(b,c, x_1,x_2)\in \mathrm{GF}(q)^4:f(x_i)+bx_i+c=u_i \ (i=1,2) \right  \}.
\end{align*}
Then
\begin{align*}
|T| = \sum_{(b,c)\in \mathrm{GF}(q)^2} J(b,c)= \sum_{(x_1, x_2)\in \mathrm{GF}(q)^2} K(x_1,x_2),
\end{align*}
where
\begin{align*}
J(b,c)=|\left  \{(x_1, x_2)\in \mathrm{GF}(q)^2 : f(x_i)+bx_i+c=u_i \ (i=1,2)\right \}|,
\end{align*}
and 
\begin{align*}
K(x_1,x_2)=|\left  \{(b, c)\in \mathrm{GF}(q)^2 : f(x_i)+bx_i+c=u_i \  (i=1,2)\right \}|.
\end{align*} 
For $J(b,c)$, we have 
\begin{align*}
J(b,c)=\begin{cases}
0, & \{u_1, u_2\} \not \subseteq V(g),\\
 1, & \{u_1, u_2\}  \subseteq V(g) \text{ and } v(g)=q,\\
 4, & \{u_1, u_2\}  \subseteq V(g) \text{ and } v(g)=q/2,
\end{cases}
\end{align*}
where $g=f(x)+bx+c$.\\
Note that
\begin{align*}
v(f(x)+bx+c)=\begin{cases}
 q, & b=0,\\
 q/2, & b\neq 0. 
\end{cases}
\end{align*}
We have
\begin{align*}
|T|= & \sum_{(b,c)\in \mathrm{GF}(q)^2} J(b,c)\\
=& | \{(b,c) \in \mathrm{GF}(q)^2 : v(f(x)+bx+c)=q\}| + 4 |I(u_1,u_2)| \\
=& q + 4 | I(u_1,u_2)|.
\end{align*} 
For $K(x_1,x_2)$, we have 
\begin{align*}
K(x_1,x_2)=\begin{cases}
0, &  x_1=x_2,\\
1, & x_1\neq x_2.
\end{cases}
\end{align*}
Thus,
\begin{align*}
|T| = & \sum_{(x_1, x_2)\in \mathrm{GF}(q)^2} K(x_1,x_2)= q(q-1).
\end{align*}
Finally, 
\begin{align*}
I(u_1,u_2)=\frac{1}{4}\left ( q(q-1) -q \right ) =\frac{q(q-2)}{4}.
\end{align*}
This completes the proof. 
\end{proof}

Another major result of this paper is the following. 

\begin{theorem}\label{thm-mar95}
Let $f(x)$ be an o-polynomial over $\gf(q)$ such that $|\cB_{(f,q/2)}|=(q-1)q$. 
Then $\bD(f, q/2):=(\gf(q), \cB_{(f,q/2)})$ is a $2$-$(q, q/2, q(q-2)/4)$ design. 
\end{theorem}

\begin{proof}
It follows from Theorem \ref{thm-opoly2to1} that $|B_{(f,b,c)}|=q/2$ for all 
$(b, c) \in \gf(q)^* \times \gf(q)$. By assumption, all blocks $B_{(f,b,c)}$ 
with $(b, c) \in \gf(q)^* \times \gf(q)$ are pairwise distinct. The design 
property then follows from Lemma \ref{lem-mar94}.  
\end{proof}  

Regarding Theorem \ref{thm-mar95}, one basic question is which of the known o-polynomials 
satisfy $|\cB_{(f,q/2)}|=q(q-1)/2$. It will be shown later that $|\cB_{(f,q/2)}|=2(q-1)$ for translation o-monomials $x^{2^h}$ and their variants $(ax)^{2^h}$. For other 
o-polynomials, we have the following conjecture, which is strongly supported by experimental 
data. 

\begin{conj} 
Let $f(x)$ be any o-polynomial over $\gf(q)$ such that $f(x) \neq (ax)^{2^h}$  
for all $a \in \gf(q)^*$ and all $h$ with $1 \leq h <m$ and $\gcd(h,m)=1$. 
Then $|\cB_{(f,q/2)}|=q(q-1)$. 
\end{conj}    

As pointed out earlier, o-polynomials do not support $3$-designs in general with respect to 
the construction of Section \ref{sec-construct}. However, 
if an o-polynomial $g(x)$ can be expressed as $(ux+v)^e + c$, where $x^e$ is an 
o-monomial, then $g(x)$ does support a $3$-design. For example, $g(x)=x^6+x^4+x^2 
=(x+1)^6 +1$. Since $x^6$ is an o-monomial over $\gf(2^m)$, where $m$ is odd, $g(x)$ 
supports a $3$-design.     

We would make the following comments on $2$-designs $\bD(f, q/2)$ supported by 
o-polynomials $f(x)$ such that $f(x) \neq (ax+b)^e + b^e$ for all o-monomials $y^e$. 
\begin{enumerate}
\item They are not $3$-designs in general. For example, when $m=5$ and $m=7$, 
the Cherowitzo o-polynomial, Payne o-polynomial, and Subiaco o-polynomial support 
only $2$-designs. 
\item The $2$-designs $\bD(f, q/2)$ from these o-polynomials are not affine-invariant, as their 
automorphism groups are smaller than the general affine group $\AG_1(\gf(q))$. 

For example, when $m=5$, the sizes of the automorphism groups of the $2$-designs 
supported by the Cherowitzo o-polynomial, Payne o-polynomial and Subiaco o-polynomial 
are $160$, while $|\AG_1(\gf(q))|=993$.   

\item These $2$-designs $\bD(f, q/2)$ cannot be isomorphic to the hyperoval $2$-designs documented 
in Theorems \ref{thm-hpodesign1} and \ref{thm-hpodesign2}, as their parameters do not 
match. 
\end{enumerate}

For the $3$-designs $\bD(f, q/2)$ supported by o-monomials, we have the following remarks. 
\begin{enumerate}
\item They are not $4$-designs according to Magma experiments. 
\item They are affine-invariant, i.e., $\AG_1(\gf(q))$ is a subgroup of their 
      automorphism groups. Experimental data indicates that their automorphism 
      groups are larger than $\AG_1(\gf(q))$. 
      
      For example, when $m=5$ and $m=7$, the automorphism groups of the $3$-designs supported 
      by the first Glynn o-monomial, second Glynn o-monomial, and the Segre o-monomial 
      have size $q(q-1)m$, while $|\AG_1(\gf(q))|=q(q-1)$. In these two cases, the automorphism 
      groups of these designs are 
      $$ 
      \GaA_1(\gf(q))=\left\{ux^{2^i}+v: (u, v) \in \gf(q)^* \times \gf(q), \ 0 \leq i \leq m-1\right\}. 
      $$ 
      The degree of transitivity of the group $\GaA_1(\gf(q))$ acting on $\gf(q)$ is only 
      $2$, and cannot be used to the prove the $3$-design property of these designs.     
      
      When $m=5$, the automorphism group 
      of the design supported by the translation o-monomial $x^2$ has size $319979520$, 
      while $|\AG_1(\gf(q))|=992$. This is a special and degenerated case, and will be 
      treated shortly.  
\item They are not symmetric designs, as only trivial $3$-designs exist. Only the designs 
      supported by the translation o-monomials are quasi-symmetric. Other $3$-designs have 
      many block intersection numbers according to experimental data.         
\end{enumerate}

\begin{open} 
Find the automorphism groups of the designs $\bD(f, q/2)$ supported by the known o-polynomials 
$f(x)$.  
\end{open}

\subsubsection{The parameters of the $3$-designs from the translation o-monomial $x^{2^h}$} 

Let $\gcd(h, m)=1$. Recall that 
$$ 
J_{2^h}=\{y^{2^h} +y: y \in \gf(q)\}. 
$$ 
Obviously, $J_{2^h}$ is an additive subgroup of $(\gf(q),+)$ with order $q/2$,  

Let $(u, v) \in \gf(q)^* \times \gf(q)$ with $uJ_{2^h}+v=J_{2^h}$. Note that 
$uJ_{2^h}$ is also an additive subgroup of $(\gf(q),+)$ with order $q/2$. It 
then follows that $J_{2^h}+v$ is also an additive subgroup of order $q/2$, 
which forces $v \in J_{2^h}$. Consequently, 
\begin{eqnarray}\label{eqn-mar991}
uJ_{2^h}=J_{2^h}. 
\end{eqnarray}

Let $J_{2^h}^*=J_{2^h} \setminus \{0\}$. It is known that $J_{2^h}^*$ is a 
Singer difference set with parameters $(q-1, (q-2)/2, (q-4)/4)$ in the 
group $(\gf(q)^*, \times)$ (see Theorem \ref{thm-maschietti}). It then follows from (\ref{eqn-mar991}) that 
$u=1$. Consequently, 
$$ 
\Stab_{\AG_1(\gf(q))}(J_{2^h})=\{x+v: v \in J_{2^h}\} 
$$  
and 
$$ 
\mu = |\Stab_{\AG_1(\gf(q))}(J_{2^h})|=q/2. 
$$ 
The following then follows from Theorem \ref{thm-mar901}. 

\begin{corollary}\label{cor-mar905} 
Let $\gcd(h, m)=1$ and $f(x)=x^{2^h}$. Then $\bD(f, q/2):=(\gf(q), \cB_{(f,q/2)})$ is a $3$-$(q, q/2, (q-4)/4)$ design. 
\end{corollary} 

Note that the number of blocks in the design of Corollary \ref{cor-mar905} is $2(q-1)$. 
Therefore, it is not a symmetric design. It is also well known that nontrivial symmetric 
$3$-designs do not exist. Below we prove that the $3$-design in Corollary \ref{cor-mar905} 
is quasi-symmetric. 

\begin{theorem}
The $3$-design of Corollary \ref{cor-mar905} has two block intersection numbers $0$ and 
$q/4$, and is thus quasi-symmetric.  
\end{theorem} 

\begin{proof}
We prove the conclusion only for odd $m$, as the proof for the other case is similar. 
Let notation be the same as before. Since $m$ is odd, $J_{2^h}$ does not contain $1$. 
In this case the block set becomes 
$$ 
\cB_{(f,q/2)}=\{uJ_{2^h} +uv: u \in \gf(q)^*, \ v \in \gf(2)  \}.  
$$ 
Let $(u_1, v_1)$ and $(u_2, v_2)$ be two elements in $\gf(q) \times \gf(2)$. Define 
$$ 
I=(u_1J_{2^h}+u_1v_1) \cap (u_2J_{2^h}+u_2v_2). 
$$ 
We now consider the value $|I|$ by distinguishing among the following cases. 

Assume that $(v_1, v_2)=(0,0)$. Then 
\begin{eqnarray*}
I = u_1J_{2^h} \cap u_2J_{2^h} = \left(u_1J_{2^h}^* \cap u_2J_{2^h}^* \right) \cup \{0\}. 
\end{eqnarray*} 
Since $J_{2^h}^*$ is a $(q-1, (q-2)/2, (q-4)/4)$ Singer difference set in $(\gf(q)^*, \times)$ 
(see Theorem \ref{thm-maschietti}), 
we have then 
\begin{eqnarray*}
|I| =  \left| \left(u_1J_{2^h}^* \cap u_2J_{2^h}^* \right) \right| + 1 
    = \left\{ 
        \begin{array}{ll}
         q/4 & \mbox{ if } u_1 \neq u_2, \\
         q/2   & \mbox{ if } u_1 = u_2. 
        \end{array}
        \right.  
\end{eqnarray*}  

Assume that $(v_1, v_2)=(0,1)$ or $(v_1, v_2)=(1,0)$. Note that $1 \not\in J_{2^h}$ and 
$uJ_{2^h}$ is an additive subgroup of $(\gf(q), +)$. It is easily seen that $I=\emptyset$. 

Finally, assume that $(v_1, v_2)=(1,1)$. We have then 
$$ 
I=u_1(J_{2^h}+1) \cap u_2(J_{2^h}+1). 
$$ 
It is known that $J_{2^h}+1$ is a $(q-1, q/2, q/4)$ Singer difference set in 
$(\gf(q)^*, \times)$. We then deduce that $|I|=q/4$ if $u_1 \neq u_2$, and $|I|=q/2$ 
otherwise. This completes the proof. 
\end{proof} 

The foregoing discussions in this section showed that the $3$-$(q, q/2, (q-4)/4)$ 
designs from the translation o-monomials $x^{2^h}$ are related to the Singer 
difference sets with parameters $(q-1, (q-2)/2, (q-4))$. It is very likely that they are isomorphic to the extended designs of the developments of the Singer difference sets with parameters $(q-1, (q-2)/2, (q-4))$. This is because every quasi-symmetric $3$-design with the block intersection number 
$0$ is the extension of a symmetric $2$-design \cite{Shri}. Anyway, our construction of the 
quasi-symmetric $3$-designs uses the direct approach $\bD(f, q/2)$, and relates the designs to translation 
hyperovals. 

\subsubsection{Parameters of the $3$-designs from other o-monomials} 

To determine the $\lambda$ value of the $3$-$(q, q/2, \lambda)$ design $\bD(x^e, q/2)$ 
from an o-monomial other than the translation o-monomials $x^e$, we need determine the 
size of the stabilizer $\Stab_{\AG_1(\gf(q))}(J_e)$ of $J_e$, both of which were defined in 
Theorem \ref{thm-mar901}. Experimental data strongly supports the next conjecture. 

\begin{conj}\label{conj-march141} 
Let $x^e$ be an o-monomial, where $e$ is not a power of $2$. Then 
$$ 
\Stab_{\AG_1(\gf(q))}(J_e)=\{x\}. 
$$ 
Consequently, the design $\bD(x^e, q/2)$ has parameters 
$3$-$(q, q/2, q(q-4)/8)$. 
\end{conj} 

To settle this conjecture, one may need the following result proved by Maschietti 
\cite{Masch98}. 

\begin{theorem}\label{thm-maschietti}
Let $e$ be a positive integer with $\gcd(e(e-1), q-1)=1$. Then $x^e$ is an o-monomial 
if and only if $J_e^*=J_e \setminus \{0\}$ is a $(q-1, (q-2)/2, (q-4)/4)$ difference set 
in $(\gf(q)^*, \times)$. 
\end{theorem} 

Below we prove Conjecture \ref{conj-march141} for several o-monomials. 
Let $J_e$ be defined in (\ref{eqn-Je2}). Define the following Boolean function 
$h(x)$ from $\mathrm{GF}(q)$ to $\gf(2)$:
\begin{align}\label{eqn-hfunction}
h(x)=\begin{cases}
1, & \mathrm{ if  }\quad x \in J_e,\\
0, & \mathrm{ otherwise}.
\end{cases} 
\end{align}

To prove Conjecture \ref{conj-march141} for several o-monomials, we need the following lemma. 

\begin{lemma}[\cite{Xiang98}]\label{lem-xiang} 
Let $m$ be odd and $e=2^i+2^j$ with $1\le i < j\le m-1$. If $f(x)=x^e$ is an o-polynomial over $\mathrm{GF}(2^m)$, then
\begin{align*}
\hat{h}(\beta)=\begin{cases}
0, &  \mbox{ if } \ \mathrm{Tr}(\beta^{\ell})=0,\\
 \pm 2^{\frac{m+1}{2}}, &   \mbox{ if } \  \mathrm{Tr}(\beta^{\ell})=1,
\end{cases}
\end{align*}
where $\hat{h}$  denotes the Walsh transform  of $h$  and
\begin{align}\label{eqn-elldef}
\ell \equiv \frac{e-1}{e} \pmod {(2^m-1)}.
\end{align} 
\end{lemma} 

By Lemma \ref{lem-GEHKX} or \ref{thm-basicproperty2}, $x^\ell$ is also an o-monomial 
over $\gf(q)$. We will make use of this fact shortly below. We now prove the following 
lemma, which settles Conjecture \ref{conj-march141} for several o-monomials over $\gf(q)$. 

\begin{lemma}\label{lem-march147} 
Let $m$ be odd and $e=2^i+2^j$ with $1\le i < j\le m-1$. Let $(b,c) \in \mathrm{GF}(q)^*\times \mathrm{GF}(q)$. If $f(x)=x^e$
is an o-polynomial over $\mathrm{GF}(q)$, then  
$$ 
\Stab_{\AG_1(\gf(q))}(J_e)=\{x\}, 
$$ 
where $J_e$ was defined in (\ref{eqn-Je2}). 
\end{lemma} 

\begin{proof}
Let $h(x)$ be defined in (\ref{eqn-hfunction}), which is the characteristic function of the set 
$J_e$. Let $(b, c) \in \gf(q)^* \times \gf(q)$ such that $h(bx+c)=h(x)$. The desired conclusion is the same 
as that $(b,c)=(1,0)$. 

Let $A=\sum_{x\in \mathrm{GF}(q)} (-1)^{h(x)+h(bx+c)}$. Since $h(bx+c)=h(x)$, we have $A=q$. 
We now compute $A$ in a different way. Note that 
\begin{eqnarray*}
\sum_{\beta \in \mathrm{GF}(q)} (-1)^{\mathrm{Tr}(\beta (x+y))} 
=\left\{ 
\begin{array}{ll}
q & \mbox{ if }  x=y, \\
0 & \mbox{ if }   x \neq y. 
\end{array}
\right. 
\end{eqnarray*} 
We have then 
\begin{align*}
q A=& \sum_{x,y\in \mathrm{GF}(q)} (-1)^{h(x)+h(by+c)} \sum_{\beta \in \mathrm{GF}(q)} (-1)^{\mathrm{Tr}(\beta (x+y))}\\
=& \sum_{\beta \in \mathrm{GF}(q)} \sum_{x\in \mathrm{GF}(q)} (-1)^{h(x)+\mathrm{Tr}(\beta x)} \sum_{y\in \mathrm{GF}(q)} (-1)^{h(by+c)+\mathrm{Tr}(\beta y)}\\
=& \sum_{\beta \in \mathrm{GF}(q)} \hat{h}(\beta) \sum_{y\in \mathrm{GF}(q)} (-1)^{h(by+c)+\mathrm{Tr}\left (\frac{\beta}{b} (by+c)+\frac{c\beta}{b} \right )}\\
=& \sum_{\beta \in \mathrm{GF}(q)} \hat{h}(\beta) \hat{h}(\frac{\beta}{b}) (-1)^{\mathrm{Tr}\left (\frac{c\beta}{b}\right )}.
\end{align*} 
Since $A=q$, we then deduce that 
\begin{eqnarray}\label{eqn-march145}
q^2=\sum_{\beta \in \mathrm{GF}(q)} \hat{h}(\beta) \hat{h}(\frac{\beta}{b}) (-1)^{\mathrm{Tr}\left (\frac{c\beta}{b}\right )}.
\end{eqnarray}
Using this equation and Lemma \ref{lem-xiang}, below we prove that $(b,c)=(1,0)$. 

Recall that $x^\ell$ is a permutation of $\gf(q)$, where $\ell$ was defined in (\ref{eqn-elldef}). 
Suppose that $b \neq 1$. Then $b^\ell \neq 1$. Consequently, the total number of 
$\beta$ in $\gf(q)$ such that $\tr(\beta^\ell)=1$ and $\tr((\beta/b)^\ell)=1$ is $2^{m-2}$. 
It then follows from Lemma \ref{lem-xiang} that 
\begin{eqnarray*}
\sum_{\beta \in \mathrm{GF}(q)} \hat{h}(\beta) \hat{h}(\frac{\beta}{b}) (-1)^{\mathrm{Tr}\left (\frac{c\beta}{b}\right )} 
&=& \sum_{\tr(\beta^\ell)=1,\ \tr((\beta/b)^\ell)=1} \hat{h}(\beta) \hat{h}(\frac{\beta}{b}) (-1)^{\mathrm{Tr}\left (\frac{c\beta}{b}\right )} \\
& \leq & 
\sum_{\tr(\beta^\ell)=1,\ \tr((\beta/b)^\ell)=1} 2^{\frac{m+1}{2}} 2^{\frac{m+1}{2}} \times 1 \\
&=& 2^{m-2} 2^{\frac{m+1}{2}} 2^{\frac{m+1}{2}} \\ 
&=& 2^{2m-1} \\ 
&<& q^2,  
\end{eqnarray*} 
which is contrary to (\ref{eqn-march145}). Consequently, we must have $b=1$. Since $b=1$, by 
Lemma \ref{lem-xiang} Equation (\ref{eqn-march145}) becomes 
\begin{eqnarray}\label{eqn-march146}
q^2=2^{m+1} \sum_{\tr(\beta^\ell)=1}  (-1)^{\mathrm{Tr}\left (c\beta\right )}.
\end{eqnarray} 
This equation forces $\tr(c\beta)=0$ for all the $2^{m-1}$ nonzero elements $\beta \in \gf(q)$ 
such that $\tr(\beta^\ell)=1$. Note that $\tr(c \times 0)=0$. Thus, $\tr(cx)=0$ has at least $2^{m-1}+1$ 
solutions, which is possible only if $c=0$. This completes the proof.  
\end{proof} 

The next result follows directly from Theorem \ref{thm-mar901} and Lemma \ref{lem-march147}. 

\begin{corollary}\label{cor-march148} 
The incidence structure $\bD(f, q/2):=(\gf(q), \cB_{(f,q/2)})$ is a $3$-$(q, q/2, q(q-4)/8)$ design if 
$f(x)=\Segre(x)$ or $f(x)=\Glynntwo(x)$.   
\end{corollary}  

It can be easily proved that $\bD(f, q/2)$ is isomorphic to $\bD(f^{-1}, q/2)$ if $f$ 
is an o-monomial over $\gf(q)$. The conclusion of Corollary \ref{cor-march148} is also 
true for the two designs $\bD(\Segre^{-1}(x), q/2)$ and $\bD(\Glynntwo^{-1}(x), q/2)$. 
Note that Conjecture \ref{conj-march141} is still open for the o-monomials 
$\overline{\Segre}(x)$ and  $\Glynnone(x)$.

It is well known that the development of the difference set $J_e^*$ can be extended 
into a $3$-$(q, q/2, (q-4)/4)$ design. For any o-monomial $x^e$, where $e$ is not a 
a power of $2$, the $3$-design $\bD(x^e, q/2)$ has parameters $3$-$(q, q/2, q(q-4)/8)$. 
Therefore, our $3$-designs $\bD(x^e, q/2)$ supported by such o-monomials $x^e$ cannot be isomorphic to the 
extended $3$-design of the development of the difference set $J_e^*$. Recall that the 
translation o-monomials are exceptions.

\subsubsection{The isomorphy of designs $\bD(f, q/2)$ from o-polynomials $f$}

First of all, we point out that two equivalent o-polynomials $f$ and $g$ may give two 
non-isomorphic designs $\bD(f, q/2)$ and $\bD(g, q/2)$. For example, by Lemma \ref{lem-GEHKX} 
the two o-monomials $x^2$ and $x^{q-2}$ are equivalent, but $\bD(x^2,q/2)$ and 
$\bD(x^{q-2}, q/2)$ 
are not isomorphic, as $\bD(x^2, q/2)$ is a $3$-$(q, q/2, (q-4)/4)$ design and 
$\bD(x^{q-2},q/2)$ is a 
$3$-$(q, q/2, q(q-4)/8)$ design. By Lemma \ref{lem-GEHKX}, the two hyperovals $\cH(x^2)$ 
and $\cH(x^{q/2})$ are equivalent, while it can be proved that the two designs $\bD(x^2, q/2)$ 
and $\bD(x^{q/2}, q/2)$ are isomorphic. Hence, the equivalence of o-polynomials is 
different from the isomorphy of designs $\bD(f, q/2)$ from o-polynomials.  

If $f(x)=x^e$ is an o-polynomial, then it is easily seen that $\bD(f, q/2)$ and 
$\bD(f^{-1}, q/2)$ are isomorphic. But $\bD(f, q/2)$ and 
$\bD(f^{-1}, q/2)$ may not be isomorphic if $f$ is not a monomial. For example, 
$\bD(\Cherowitzo(x), q/2)$ and 
$\bD(\Cherowitzo^{-1}(x), q/2)$ are not isomorphic when $m=5$.  

Since it is hard to do a theoretical isomorphy classification of designs $\bD(f,q/2)$ 
from o-polynomials $f$, we have done an isomorphy classification for the following set 
of o-polynomials for the case $m=5$ with Magma:  
\begin{eqnarray*}
\{
\Segre(x), \ 
\overline{\Segre}(x), \ 
\Glynnone(x), \ 
\Glynntwo(x), \ 
\Cherowitzo(x), \\  
\overline{\Cherowitzo}(x), \ 
\Cherowitzo^{-1}(x), \ 
\Payne(x), \ 
\Subiaco_1(x)
\}.  
\end{eqnarray*}
Their designs $\bD(f, q/2)$ for $m=5$ are pairwise not isomorphic, except that  
$\bD(\Segre(x), q/2)$ and $\bD(\Glynnone(x), q/2)$ are isomorphic. But $\bD(\Segre(x), q/2)$ and $\bD(\Glynnone(x), q/2)$ are not isomorphic when $m=7$. Hence, the 3-designs of these 
o-monomials are paiwise not isomorphic in general.

\section{Designs from special polynomials over $\gf(q)$ for odd $q$}\label{sec-designgfq}

Let $q$ be odd throughout this section. Theorem \ref{thm-380} says that any permutation monomial 
$x^e$ over $\gf(q)$ supports $2$-designs. Since $x^e$ is a permutation, $e$ must be odd. 
Let $d=\gcd(e-1, q-1)$. Then $d \geq 2$. 

An interesting case is that $d=\gcd(e-1, q-1)=2$. In this case, it can be 
shown that there are at most two block sizes $|\{y^e+y: y \in \gf(q)\}|$ and 
$|\{y^e+\alpha y: y \in \gf(q)\}|$, where $\alpha$ is a generator of $\gf(q)^*$.   
In this case, $x^e$ supports at most two nontrivial $2$-designs with different block sizes. 

Motivated by the foregoing discussions, we call a monomial $x^e$ over $\gf(q)$ 
for odd $q$ a t-monomial (i.e., twin design monomial) if $\gcd(e, q-1)=1$ 
and $|\{x^e+bx: x \in \gf(q)\}|$ takes only two distinct values for all 
$b \in \gf(q)^*$.   

\begin{theorem}
Let $p$ be odd and $m \geq 2$. Below is a list of monomials $x^e$ over $\gf(p^m)$ 
such that $\gcd(e, p^m-1)=1$ and $\gcd(e-1, p^m-1)=2$. 
\begin{itemize}
\item $e=3$ and $p=3$. 
\item $e=3$, $p \equiv 5 \pmod{6}$ and $m$ is odd. 
\item $e=5$, $p \in \{3,7\}$ and $m$ is odd. 
\item $e=p^m-2$. 
\item $e=(p^m-3)/2$, $p \equiv 1 \pmod{4}$ and $m$ is even, or $p \equiv 3 \pmod{4}$. 
\item $e=p^m-p-1$ and $m$ is odd.   
\end{itemize}
\end{theorem} 

\begin{proof}
It is straightforward to prove the desired conclusions for the values $e$. The deatils 
are left to the reader. 
\end{proof}

These are candidates of t-monomials. But it may be technical to prove that they 
are t-monomials. In fact, the Dickson permutation monomial $x^5$ over $\gf(3^5)$ is in 
fact a d-monomial, as the block size $|B_{(x^5, b, c)}|$ is $153$ for all $(b, c) 
\in \gf(3^5)^* \times \gf(3^5)$. 

Monomials $x^e$ with $\gcd(e, q-1) \neq 1$ may also support $2$-designs. For example, 
$x^2$ over $\gf(3^m)$ supports a $2$-$(3^m, (3^m+1)/2, (3^m+1)/4)$ symmetric design, 
which is the development of the difference set defined by all the squares in $\gf(3^m)$. 
We have also the following conjecture. 

\begin{conj}\label{conj-121} 
Let $m \geq 3$ be odd. Define $k_{m}$ by the recurrence relation 
$$ 
k_m = \frac{3^m+1}{2} + 3^{m-1} - 3 k_{m-2}  
$$ 
with initial vale $k_1=2$. Then 
\begin{eqnarray*}
\left| B_{(x^{10}, b, c)} \right|= 
\left\{
\begin{array}{ll}
\frac{3^m+1}{2}  & \mbox{ with $3^m$ times,} \\
k_m              & \mbox{ with $(3^m-1)3^m$ times}
\end{array}
\right.  
\end{eqnarray*} 
and 
$$ 
|\cB_{(x^{10}, (3^m+1)/2)}|=3^m, \ |\cB_{(x^{10},k_m)}|=\frac{3^m(3^m-1)}{2}.  
$$ 
Further, 
\begin{itemize}
\item $(\gf(3^m), \cB_{(x^{10}, (3^m+1)/2)})$ is a $2$-$(3^m, \, (3^m+1)/2, \, (3^m+1)/4)$ symmetric 
design, which is the development of the difference set consisting of all the squares 
in $\gf(3^m)$; and   
\item $(\gf(3^m), \cB_{(x^{10}, k_m)})$ is a $2$-$(3^m, \, k_m, \, k_m(k_m-1)/2)$ design. 
\end{itemize}
\end{conj} 

If Conjecture \ref{conj-121} is true, the design $(\gf(3^m), \cB_{(x^{10}, k_m)})$ would be 
interesting. The following is a fundamental open problem.  

\begin{open} 
Is there a polynomial $f(x)$ over $\gf(q)$ with odd $q$ such that $\bD(f, k)$ is 
a $3$-design for some $k$? 
\end{open}

\section{An extended construction of $t$-designs from polynomials}\label{sec-extdendedconstr} 

In the construction of designs introduced in Section \ref{sec-construct}, 
not every polynomial $f$ supports a $2$-design $\bD(f, k)$. Only special polynomials 
over $\gf(q)$ can support a $2$-design. In this section, we outline an extended  
construction of $2$-designs from polynomials over finite fields $\gf(q)$. 

Let $f(x)$ be a polynomial over $\gf(q)$. For each $(a,b,c) \in \gf(q)^3$, we define 
\begin{eqnarray}
\hat{B}_{(f,a,b,c)}=\{af(x)+bx+c: x \in \gf(q)\}.  
\end{eqnarray} 
Let $k$ be any integer with $2 \leq k \leq q$. Define 
\begin{eqnarray}
\hat{\cB}_{(f,k)}=\{\hat{B}_{(f,a,b,c)}: |\hat{B}_{(f,a,b,c)}|=k, \ (a,b,c) \in \gf(q)^3\}.  
\end{eqnarray} 
We have then the following result. 

\begin{theorem}\label{thm-march131}
Let notation be the same as before. If $|\hat{\cB}_{k}|>1$, then the incidence structure 
$\hat{\bD}(f, k)=(\gf(q), \hat{\cB}_{(f,k)})$ is a $2$-$(q^m, k, \lambda)$ design for some 
$\lambda$. 
\end{theorem} 

\begin{proof}
The desired conclusion follows from the facts that the general affine group $\GA_1(\gf(q))$ 
is a subgroup of the automorphism group of the incidence structure $\hat{\bD}(f, k)$, 
$\GA_1(\gf(q))$ fixes $\hat{\cB}_{(f,k)}$,  and 
$\GA_1(\gf(q))$ acts on $\gf(q)$ doubly transitively.  
\end{proof}

Theorem \ref{thm-march131} tells us that almost every polynomial over $\gf(q)$ gives 
$2$-designs under this extended construction $\hat{\bD}(f, k)$. This fact makes the 
extended construction $\hat{\bD}(f,k)$ less interesting than the previous one $\bD(f,k)$, 
though many $2$-designs with nice parameters may be obtained by choosing 
special types of polynomials. However, it would be very nice if this extended construction 
$\hat{\bD}(f,k)$ can produce $t$-designs with $t \geq 3$. 

It is easily seen that for any o-monomial $x^e$ over $\gf(2^m)$, we have 
$\hat{\bD}(x^e, 2^{m-1})=\bD(x^e, 2^{m-1})$. Hence, it is indeed a $3$-design, 
but was already covered by the construction $\bD(x^e, 2^{m-1})$. 

Recall that $\bD(f, 2^{m-1})$ is only a $2$-design if $f$ is the Cherowitzo or Payne 
trinomial. What will happen if we plug the Cherowitzo and Payne trinomials 
into this extended construction? Regarding this question, we have the following. 

\begin{theorem}\label{thm-march136}
Let $m \geq 4$ and $q=2^m$. Then the incidence structure $\hat{\bD}(f, q/2)=(\gf(q), \hat{\cB}_{(f,q/2)})$ is a 
$3$-$(q, q/2, (q-4)(q-1)q/8)$ design if $f$ is an o-polynomial over $\gf(q)$ with 
$|\hat{\cB}_{(f,q/2)}|=q(q-1)^2$. 
\end{theorem} 

\begin{proof}
Lemmas \ref{lem-mar91} and \ref{lem-mar92} can be modified into a proof of the 
desired result. The details are omitted.  
\end{proof}

Theorem \ref{thm-march136} is valuable only when there is an o-polynomial over $\gf(q)$ with 
$|\hat{\cB}_{(f,k)}|=q(q-1)^2$. In fact, we have the following conjecture.   

\begin{conj}\label{con-march132} 
Let $m \geq 5$ be odd and $q=2^m$. Let $f(x)$ be an o-polynomial over $\gf(q)$ such 
that $f(x) \neq (ax+b)^e+b^e$ for all o-monomials $y^e$ and all $(a,b) \in \gf(q)^2$. 
Then $|\hat{\cB}_{(f,q/2)}|=q(q-1)^2$. 
\end{conj}   

It might be hard to settle Conjecture \ref{con-march132} in general. But it is possible to prove 
the conjecture for the Cherowitzo, Payne and Subiaco o-polynomials. The reader is 
invited to attack this conjecture. 

We inform the reader that Conjecture \ref{con-march132} is indeed true for the Cherowitzo trinomial, 
Payne trinomial and Subiaco polynomials for $m \in \{5, 7, 9\}$ according to Magma experimental 
data. Hence, $3$-designs have been indeed obtained from this extended construction 
$\hat{\bD}(f, q/2)$ with o-polynomials introduced in this section. Recall that 
$\hat{\bD}(f, q/2)$ is always a $2$-design for any o-polynomial $f$ over $\gf(q)$ by  
Theorem \ref{thm-march136}, and a $3$-design for any o-monomial over $\gf(q)$,  
where $q=2^m$.

\section{Summary and concluding remarks}\label{sec-summary} 

The main contributions of this paper are the following: 
\begin{enumerate}
\item The first one is the two general constructions of $t$-designs with polynomials over 
      finite fields 
      documented in Sections \ref{sec-construct} and \ref{sec-extdendedconstr}. Many 
      types of polynomials may be plugged into the two constructions for obtaining 
      $t$-designs with different parameters. 
\item The second is the application of o-polynomials in $t$-designs under the frameworks 
      of the two 
      general constructions. The first construction has produced infinite families of 
      $3$-designs from o-monomials over $\gf(2^m)$, and infinite families of $2$-designs 
      from o-polynomials over $\gf(2^m)$. The second construction has given $2$-designs 
      and also $3$-designs from o-polynomials over $\gf(2^m)$.          
\end{enumerate} 

Some of the $2$-designs obtained in this paper are affine-invariant, while other 
$2$-designs are not affine-invariant and thus interesting. Some $3$-designs presented 
in this paper are indeed affine-invariant, but their automorphism groups are only doubly 
transitive on their point sets. So the $3$-design property of these designs had to be 
proved with direct approaches. This makes these $3$-designs very special. 

Since the two constructions of $t$-designs are quite general, a lot of work can be done 
in this direction. Several open problems and conjectures were presented in this paper. 
The reader is cordially invited to join the venture into the topic of this paper.



\end{document}